\documentclass[reqno]{amsart}

\usepackage{mymacro}
\usepackage{bookmark}

\usepackage{fancyhdr}

\setlist[enumerate, 1]{label = \roman*., ref = \roman*}
\setlist[enumerate, 2]{label = \theenumi.\alph*}

\usepackage{hyperref}
\usepackage{url}
\hypersetup{breaklinks=true}

\makeatletter
\@namedef{subjclassname@2020}{%
  \textup{2020} Mathematics Subject Classification}
\makeatother

\author[Gilabert]{Mart\'in Gilabert Vio}
\address{\hskip-\parindent{}
	Institut Camille Jordan, Université Claude Bernard Lyon 1, 43 boulevard du 11 novembre
1918, 69622 Villeurbanne, France.}
\email{gilabert@math.univ-lyon1.fr // martingilabertvio@gmail.com}
\date{\today}


\title{Probabilistic Tits alternative for circle diffeomorphisms}

\begin{document}

\begin{abstract}
Let $\mu_1, \mu_2$ be probability measures on $\mathrm{Diff}^1_+(S^1)$ satisfying a suitable moment condition and such that their supports genererate discrete groups acting proximally on $S^1$. Let $(f^n_\omega)_{n \in \N}, (f^n_{\omega'})_{n \in \N}$ be two independent realizations of the random walk driven by $\mu_1, \mu_2$ respectively. We show that almost surely there is an $N \in \N$ such that for all $n \geq N$ the elements $f^n_\omega, f^n_{\omega'}$ generate a nonabelian free group. The proof is inspired by the strategy by R. Aoun for linear groups and uses work of A. Gorodetski, V. Kleptsyn and G. Monakov, and of P. Barrientos and D. Malicet. A weaker (and easier) statement holds for measures supported on $\mathrm{Homeo}_+(S^1)$ with no moment conditions.
\end{abstract}

\maketitle

\section{Context and contributions}
The Tits alternative is a celebrated theorem by J. Tits which asserts that finitely generated linear groups are either virtually solvable or contain a nonabelian free group \cite{titsFree}. This alternative fails for groups of homeomorphisms of the circle, but a weaker alternative (sometimes called a \emph{dynamical Tits alternative}, see \cite{malicetMiliton}) still holds.
\begin{teo}[G. Margulis \cite{margulisCircle}, see also V. Antonov \cite{antonov}] \label{teo:margulis}
Let $G$ be a subgroup of $\mathrm{Homeo}(S^1)$. Then either
\begin{enumerate}
	\item $G$ is \emph{elementary}, that is, the action of $G$ preserves a probability measure on $S^1$, or
	\item \label{it:M2} $G$ contains a \emph{ping-pong pair}, that is, two elements $f,g \in G$ such that there are pairwise disjoint open subsets $U_1, U_2, V_1, V_2$ of $S^1$ with $f(S^1 - U_1) \subseteq V_1,\, g(S^1 - U_2) \subseteq V_2$.
\end{enumerate}
\end{teo}

The previous options are mutually exclusive and, by the ping-pong lemma, condition \eqref{it:M2} implies that $f,g$ generate a nonabelian free group in $G$. We are interested in how generic these two elements are. As a first approximation, there is a dense $G_\delta$ subset $W$ of $\mathrm{Homeo}(S^1) \times \mathrm{Homeo}(S^1)$ such that any pair of elements in $W$ generate a nonabelian free group \cite[Proposition 4.5]{ghysCircle} (see also \cite[Theorem 6.9]{triestinoBook}). Our viewpoint will be probabilistic instead of topological, inspired by the following result of R. Aoun for linear groups. We first fix some notation: a probability measure $\mu$ on a group $G$ is said to be \emph{nondegenerate} if the semigroup generated by its support is all $G$. Given nondegenerate probability measures $\mu_1, \, \mu_2$ on groups $G_1,\, G_2$ we let $(\Omega_i,\P_i),\, i = 1,2$ be the probability space $(G_i^\N, \mu_i^{\otimes \N})$, and we write $\omega = (f_{\omega_n})_{n \in \N}$ for an element of $\Omega_1$ and $\omega' = (f_{\omega'_n})_{n \in \N}$ for an element of $\Omega_2$. Also, denote $f^n_\omega$ for the right random walk $f_{\omega_n} \circ f_{\omega_{n-1}} \circ \cdots \circ f_{\omega_0}$ at time $n \in \N$.

\begin{teo}[R. Aoun \cite{aoun, aoun1}] \label{teo:aoun}
Let $G$ be a real algebraic linear group that is semisimple and with no compact factors, and let $G_1,\, G_2$ be Zariski-dense subgroups of $G$. If $\mu_1, \mu_2$ are nondegenerate probability measures on $G$ with finite exponential moment, then there exists a $\rho \in (0,1)$ such that
\[
	\P_1\otimes \P_2\left[ (\omega,\omega') \in \Omega_1 \times \Omega_2 \text{ such that } f^n_\omega,f^n_{\omega'} \text{ are a ping-pong pair}\, \right] \geq 1 -\rho^n
\] for all sufficiently large $n \in \N$.
\end{teo}

In the previous statement, two elements $f,g \in \mathrm{GL}_d(\R), d \in \N$ are said to be a \emph{ping-pong pair} if the conditions in \eqref{it:M2} hold for their natural action on projective space $\mathrm{P}\R^d$, and the measure $\mu$ is said to have \emph{exponential moment} if $\int_G \norm{g}^\delta \dd \mu(g)$ is finite for some $\delta > 0$ (here $\norm{ \cdot }$ is any norm on $d \times d$ matrices).

When specialized to $\mathrm{PSL}_2(\R)$ acting on $S^1$, the proof shows that the situation depicted in Figure \ref{fig:model} occurs with probability converging to 1 exponentially fast in $n \in \N$. That is, there exist disjoint intervals $I_{n,\omega}, I_{n,\omega'}, J_{n,\omega},J_{n.\omega'} \subset S^1$ that testify that $f^n_\omega, f^n_{\omega'}$ are a ping-pong pair. The intervals $I_{n,\omega}, I_{n,\omega'}$ can be taken centered around $f^n_\omega(x), f^n_{\omega'}(y)$ where $x,y \in S^1$ are arbitrary and fixed beforehand. The intervals $J_{n,\omega},J_{n,\omega'}$ converge as $n$ increases to the \emph{repulsors} $\sigma(\omega), \sigma(\omega')$ of the random walks $f^n_\omega, f^n_{\omega'}$ (see the next section for the definition of $\sigma$). To control the probabilities that they intersect, their diameters decrease to 0 exponentially fast in $n$.

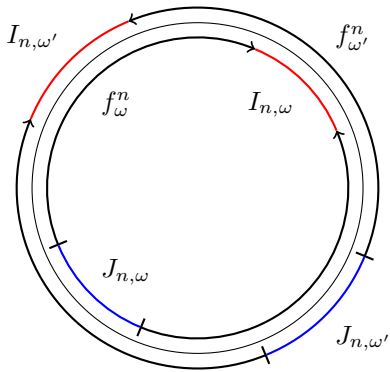
\begin{figure}[H]
\centering
\begin{tikzpicture}[scale=1]
\draw (0,0) circle (2.2);

\draw[thick,red] ([shift=(22.5:2)]0,0) arc (22.5:90-22.5:2) node[black, below left, pos=.5] {$I_{n,\omega}$};
\draw[thick,blue] ([shift=(180+22.5:2)]0,0) arc (180+22.5:270-22.5:2) node[black, above right, pos=.5] {$J_{n,\omega}$};

\draw[thick,<-|] ([shift=(90-22.5:2)]0,0) arc (90-22.5:180+22.5:2) node[below right, pos=.5] {$f^n_\omega$};
\draw[thick,<-|] ([shift=(22.5:2)]0,0) arc (22.5:-90-22.5:2);

\draw[thick,red] ([shift=(90+22.5:2.4)]0,0) arc (90+22.5:180-22.5:2.4) node[black, above left, pos=.5] {$I_{n,\omega'}$};
\draw[thick,blue] ([shift=(270+22.5:2.4)]0,0) arc (270+22.5:360-22.5:2.4) node[black, below right, pos=.5] {$J_{n,\omega'}$};

\draw[thick,|->] ([shift=(-22.5:2.4)]0,0) arc (-22.5:90+22.5:2.4) node[above right, pos=.5] {$f^n_{\omega'}$};
\draw[thick,<-|] ([shift=(180-22.5:2.4)]0,0) arc (180-22.5:270+22.5:2.4);
\end{tikzpicture}
\label{fig:model}
\end{figure}

The main result of this paper shows that that this situation remains typical for a pair of independent random walks on countable subgroups of $\mathrm{Diff}^1_+(S^1)$, the group of orientation-preserving diffeomorphisms of $S^1$, provided the action of the subgroups on $S^1$ is proximal. This condition is almost always fulfilled for a group acting on $S^1$ admitting no invariant measures on $S^1$, see Theorem \ref{teo:classification} below for a precise statement. For a function $\phi \colon S^1 \to \R$, denote
\[
	\abs{\phi}_{\mathrm{Lip}} = \sup_{x \neq y \in S^1}\frac{\abs{\phi(x) - \phi(y)}}{d(x,y)}.
\]

\begin{thmA} \label{teo:a}
Let $G_1,\, G_2$ be countable subgroups of $\mathrm{Diff}^1_+(S^1)$ such that the actions of $G_1$ and of $G_2$ on $S^1$ are proximal. Let $\mu_1, \, \mu_2$ be nondegenerate probability measures on $G_1, \, G_2$ respectively such that there exists a $\delta > 0$ so that the integral
\[
	\int_{G_i} \max\left\{\abs{g}_{\mathrm{Lip}}, \abs{g^{-1}}_{\mathrm{Lip}}\right\}^\delta \dd \mu(g)
\] is finite for $i = 1,2$.

Then there exists a $\rho \in (0,1)$ such that
\[
	\P_1\otimes \P_2\left[ (\omega,\omega') \in \Omega_1 \times \Omega_2 \text{ such that } f^n_\omega,f^n_{\omega'} \text{ are a ping-pong pair}\, \right] \geq 1 -\rho^n
\] for all sufficiently large $n \in \N$.
\end{thmA}

As with Theorem \ref{teo:aoun}, the Borel-Cantelli lemma immediately implies the following.

\begin{corA} \label{cor:cs}
Let $\mu_1, \mu_2$ be probability measures on $\mathrm{Diff}^1_+(S^1)$ satisfying the same assumptions as in Theorem \ref{teo:a}. For $\P_1\otimes \P_2$-almost every $(\omega, \omega')$ there exists an $N \in \N$ such that $f^n_\omega, f^n_{\omega'}$ generate a nonabelian free group for all $n \geq N$.
\end{corA}

The conclusion of Theorem \ref{teo:a} is known to be true in other settings: if $M$ is a proper hyperbolic space such that $\mathrm{Isom}(M)$ acts cocompactly on $M$ and $\mu$ is a measure on $\mathrm{Isom}(M)$ generating a nonelementary group, then this is \cite[Theorem 1.10]{aounSert2022}. The case of nonelementary hyperbolic groups acting on their Gromov boundary and finitely supported $\mu$ was treated previously in \cite{gilmanMiasnikovOsin2010}.

We do not know if the statement in Corollary \ref{cor:cs} is true for groups of biLipschitz homeomorphisms of $S^1$.\footnote{The relevance of the biLipschitz condition comes from the fact that any countable subgroup of $\mathrm{Homeo}_+(S^1)$ is conjugated to a group of biLipschitz homeomorphisms, see \cite[Théorème D]{deroinKleptsynNavas}.} To put this in perspective, we show that a weakening of Corollary \ref{cor:cs} is true for groups of homeomorphisms of $S^1$, even without moment assumptions on the measures $\mu_i$ and relaxing the proximality assumption on the $G_i$ to ``no invariant measures''. It is an application of results in \cite{deroinKleptsynNavas}, but has not appeared previously in the literature up to our knowledge. Denote by $\mathrm{Homeo}_+(S^1)$ the group of orientation-preserving homeomorphisms of $S^1$.

\begin{thmA}\label{teo:b}
Let $G_1,\, G_2$ be countable subgroups of $\mathrm{Homeo}_+(S^1)$ such that the actions of $G_1$ and of $G_2$ on $S^1$ do not admit invariant probability measures, and let $\mu_1, \mu_2$ be nondegenerate probability measures on $G_1,\, G_2$ respectively. Then for $\P_1 \otimes \P_2$-almost every $(\omega, \omega') \in \Omega_1 \times \Omega_2$, the set of $n \in \N$ such that $f^n_\omega,\, f^n_{\omega'}$ are a ping-pong pair has density 1, that is,
\[
	\lim_{N \to \infty}\frac{1}{N}\abs{\{0\leq n \leq N \mid f^n_{\omega},\, f^n_{\omega'} \text{ are a ping-pong pair}\}} = 1.
\]
\end{thmA}

The proof of Theorem \ref{teo:b} requires only the tools developped in \cite[Appendice]{deroinKleptsynNavas} and general statements on contracting random dynamical systems from \cite{malicet}. In contrast, the proof of Theorem \ref{teo:a} requires more ingredients. For instance, to apply the strategy of \cite{aoun} in this context it is essential to know that exponential contractions occur in mean and that the stationary measure is Hölder continuous (see Theorems \ref{teo:mean} and \ref{teo:holder} below respectively). The first condition has been already proven in different situations in the literature by K. Gelfert and G. Salcedo \cite{GScontracting}, A. Gorodetski and V. Kleptsyn \cite{KG}, and P. Barrientos and D. Malicet \cite{BarrientosMalicet2024}, all of which require (at least) that $\mu$ be supported on $\mathrm{Diff}_+^1(S^1)$. The second one is a very general theorem by A. Gorodetski, V. Kleptsyn and G. Monakov \cite{holderMeasure}. One important difference with the linear setting lies in the dynamics of individual elements of $\mathrm{Homeo}_+(S^1)$: very ``contracting'' homeomorphisms of the circle do not have a canonically defined repulsor or attractor. Proposition \ref{prop:contraccionExp} below deals with this issue.

\subsection*{Acknowledgements}
The author wishes to thank K. Gelfert for answering questions on \cite{gelfertSalcedo2024}, V. Kleptsyn for useful conversations around this subject and B. Deroin for pointing out \cite{aoun}. The author also thanks his advisor N. Matte Bon for constant encouragement and useful advice.

\section{Preliminaries} \label{sec:preliminaries}
We review some basic theory on random dynamical systems and groups acting on the circle, and introduce some notation. For more details on this material, see \cite{ghysCircle, deroinKleptsynNavas, malicet}.

\subsection*{Notation}
Given a metric space $(X,d)$, $A \subseteq X$ and $\epsilon >0$, we denote $A^\epsilon = \{x \in X\mid d(x,A) \leq \epsilon\}$. We will denote by $d$ the metric on $S^1$ coming from an identification $S^1 = \R/\Z$, so $\mathrm{diam}(S^1) = 1/2$.

\subsection*{Random dynamical systems}
A \emph{random dynamical system} $(G,\mu) \curvearrowright X$ (or a \emph{random walk on $X$}) is the data of a group $G$ acting by homeomorphisms on a compact metric space $(X,d)$ and of a probability measure $\mu$ on $G$. We always assume that $\mu$ is \emph{nondegenerate}, that is, that the semigroup generated by $\mu$ is $G$. Denote by $(\Omega, \P)$ the probability space $(G^\N, \mu^{\otimes \N})$ and set $f^n_\omega = f_{\omega_n} \circ \cdots \circ f_{\omega_0}$ when $n\in \N, \,  \omega = (f_{\omega_k})_{k \in \N} \in \Omega$. A \emph{$\mu$-stationary measure} is a Borel probability measure on $X$ such that $\mu \ast \nu = \nu$, where
\[
	\mu\ast \nu(A) = \int_G \nu(g^{-1}A) \dd \mu(g)
\] for all Borel $A \subseteq X$.

We say that $(G,\mu) \curvearrowright X$ is \emph{locally contracting} if for all $x \in X$, $\P$-almost surely there exists a neighborhood $B \subset X$ of $x$ such that $\mathrm{diam}(f^n_\omega(B)) \xrightarrow[n \to \infty]{}0$.

\begin{prop}[{\cite[Propositions 4.8 and 4.9]{malicet}}] \label{prop:medidasEstacionarias}
Suppose $(G, \mu) \curvearrowright X$ is locally contracting. Then there are finitely many ergodic $\mu$-stationary measures $\nu_1,\ldots, \nu_d$, and their supports are exactly the minimal $G$-invariant sets in $X$.

Moreover, for every $x \in X$ and $\P$-almost every $\omega$ there exist a unique index $i = i(\omega, x) \in \{1,\ldots, d\}$ such that $f^n_\omega(x)$ \emph{equidistributes towards }$\nu_i$, that is
\begin{equation} \label{eq:convDebil}
	\frac{1}{N}\sum_{n = 0}^N 1_{f^n_\omega(x)} \xrightarrow[N \to \infty ]{} \nu_i
\end{equation}
in the weak $\ast$-topology.
\end{prop}

\subsection*{Groups acting on the circle}
We say that a group action $G \curvearrowright S^1$ is \emph{proximal} if for all $x, y \in S^1$ and $\epsilon > 0$ there exists $g \in G$ such that $d(g(x), g(y)) < \epsilon$. The action is said to be \emph{locally proximal} if it is not proximal and every $x \in S^1$ is the endpoint of an interval $I \subset S^1$ such that for all $\epsilon >0$ there exists $g \in G$ with $\mathrm{diam}(g(I))< \epsilon$. In this context, we say that a group action $G \curvearrowright^\phi S^1$ is \emph{semiconjugate} to $G \curvearrowright^\psi S^1$ if there exists a continuous surjection $\pi \colon S^1 \to S^1$ such that $\psi(g) \circ \pi = \pi \circ \phi(g)$ for all $g \in G$, and such that $\pi$ is locally non-decreasing and has degree one (this means that any lift of $\pi$ to $\widetilde{\pi}\colon \R \to \R$ is non-decreasing and satisfies $\widetilde{\pi}(x + 1) = \widetilde{\pi}(x) + 1$ for all $x \in \R$). When such a $\pi$ does not necessarily have degree 1, we call it a \emph{factor map}.

The following theorem is essentially equivalent to Theorem \ref{teo:margulis} from the introduction: in cases \eqref{item:class22} and \eqref{item:class23} below $G$ always contains a free group, and there exists an invariant probability measure for $G$ in cases \eqref{item:class1} (a mean of Dirac measures on a finite orbit on $S^1$) and \eqref{item:class21} (the image of Lebesgue measure under a conjugacy to a group of rotations).

\begin{teo}[see \cite{ghysCircle}] \label{teo:classification}
Consider an action $G \curvearrowright^\phi S^1$ by orientation-preserving homeomorphisms. Then exactly one of the following options is satisfied.
\begin{enumerate}
	\item \label{item:class1} There exists a finite orbit.
	\item \label{item:class2} There exists a unique closed minimal set $\Lambda$, which is either $S^1$ or a Cantor set. In the latter case, by collapsing the countably many connected components of $S^1 - \Lambda$ we can semiconjugate $\phi$ to a minimal group action $G \curvearrowright S^1$.
\end{enumerate}
Moreover, in the minimal case a further distinction exists: either
\begin{enumerate}[label = ii.\alph*]
	\item \label{item:class21} the action is free and thus conjugated to an action by rotations, or
	\item \label{item:class22} the action is proximal, or
	\item \label{item:class23} the action is locally proximal and not proximal, and there exists $d \in \N,\, d \geq 2$ and a continuous $d$-to-one covering $\pi \colon S^1 \to S^1$ that intertwines $\phi$ with a proximal action.
\end{enumerate}
\end{teo}

Thus whenever $G \curvearrowright^\phi S^1$ does not preserve a probability measure on $S^1$, there exists $d \in \N$ and a factor map $\pi \colon S^1 \to S^1$ that is $d$-to-one on the minimal set of $G$ (except for a countable number of points), and that intertwines $\phi$ with a proximal and minimal action. We will call this integer $d$ the \emph{degree of proximality} of $\phi$, but this notation is not standard.

\subsection*{Random walks on $S^1$}
In this subsection, fix a countable group $G$ and a nondegenerate probability measure $\mu$ on $G$. The random walk on $S^1$ defined by a proximal group action $G \curvearrowright^\phi S^1$ has been well studied.

\begin{teo}[{\cite[Appendice]{deroinKleptsynNavas}}] \label{teo:dkn}
Consider an action $G \curvearrowright S^1$ by orientation-preserving homeomorphisms with no invariant probability measure on $S^1$.
\begin{enumerate}
	\item \label{it:dkn1} There exists a unique $\mu$-stationary probability measure $\nu$ on $S^1$, which is atomless and is supported on the minimal set of $G$.
	\item \label{it:dkn2} If the action of $G$ on $S^1$ is proximal, there exists a random variable $\omega \in \Omega \mapsto \sigma(\omega) \in S^1$ such that for $\P$-almost every $\omega$ we have
	\[
		(f^n_\omega)^{-1} \nu \xrightarrow[n\to \infty]{} 1_{\sigma(\omega)}
	\] in the weak-$\ast$ topology.
\end{enumerate}
\end{teo}

We call the random variable $\sigma(\omega)$ from the previous theorem the \emph{repulsor} of the random walk $(f^n_\omega)_{n \in \N}$. Its distribution is the unique $\overline{\mu}$-stationary measure on $S^1$ where $\overline{\mu} \in \mathrm{Prob}(G)$ is defined on $g \in G$ as $\mu(g^{-1})$, and is thus nonatomic. 

Notice that the measure $(f^n_\omega)^{-1}\nu$ is given by $\nu(f^n_\omega(I))$ on every interval $I \subseteq S^1$, so the statement in \eqref{it:dkn2} says that $f^n_\omega(I)$ is contracted into a $\nu$-null set unless $I$ contains $\sigma(\omega)$, in which case it is expanded to the whole circle. As a consequence, for all $x, y \in S^1$ we have $\P$-almost surely that $x$ and $y$ are not $\sigma(\omega)$ since the law of $\sigma(\omega)$ is nonatomic, and hence $\lim_{n \to \infty} d(f^n_\omega(x), f^n_\omega(y)) = 0$. This conclusion is the subject of \cite{kleptsynNalskii2004} (see also \cite{antonov}), and justifies the fact that we will use $f^n_\omega(0)$ (or $f^n_\omega(x)$ for some nonrandom $x \in S^1$) as an ``attractor'' for $f^n_\omega$ in the proofs below.

When $G \curvearrowright S^1$ is not necessarily proximal a similar statement holds, and even more is true: the rate of contraction of $f^n_\omega(I)$ when $\sigma(\omega) \not \in I$ is exponential.

\begin{teo} \label{teo:MKKO}
Consider an action $G \curvearrowright S^1$ by orientation-preserving homeomorphisms with no invariant probability measure on $S^1$. Let $d \in \N$ be the degree of proximality of $G \curvearrowright S^1$.

There exist measurable functions $\sigma_1,\ldots, \sigma_d \colon \Omega \to S^1$ such that the following hold.
\begin{enumerate}
	\item \label{it:MKKO1} \cite[Theorem A]{malicet} There exists a $\lambda > 0$ such that for $\P$-almost every $\omega$ and every closed interval $I \subset S^1 \setminus \{\sigma_1(\omega), \ldots, \sigma_d(\omega)\}$ we have 
	\[
		\mathrm{diam}(f^n_\omega(I)) \leq e^{-\lambda n}
	\] for sufficiently large $n \in \N$.
	\item \label{it:MKKO3} \cite[Proposition 4.3]{malicet} $\P$-almost surely, the set $\{\sigma_1(\omega),\ldots, \sigma_d(\omega)\}$ has size $d$.
\end{enumerate}
\end{teo}

The random set $\{\sigma_1,\ldots, \sigma_d\}$ from the previous theorem is called the \emph{repelling set} of the random walk $(f^n_\omega)_{n \in \N}$. In this setting, let $\pi \colon S^1 \to S^1$ be a factor map to a minimal and proximal action and $\Lambda \subseteq S^1$ be the minimal set of $G \curvearrowright S^1$.  Define $E \subset S^1$ as the countable set of images of connected components of $S^1 \setminus \Lambda$, so for all $x \in S^1 \setminus E$, the fiber $\pi^{-1}(x)$ has size $d$. Denote by $\sigma(\omega)$ the repulsor of the random walk in the image of $\pi$. If $\sigma(\omega) \in S^1 \setminus E$ (which happens $\P$-almost surely, since the distribution of $\sigma$ is nonatomic), then $\pi^{-1}(\sigma(\omega)) = \{\sigma_1(\omega),\ldots. \sigma_d(\omega)\}$ by the defining properties of $\sigma(\omega)$ and $\{\sigma_1(\omega),\ldots, \sigma_d(\omega)\}$. We record this as a proposition.

\begin{prop} \label{prop:repellingSet}
Consider an action $G \curvearrowright S^1$ by orientation-preserving homeomorphisms with no invariant probability measure on $S^1$. Denote by $\pi \colon S^1 \to S^1$ a factor map to a minimal and proximal action, and by $\omega \mapsto \sigma(\omega)$ the repulsor of the random walk induced in the image of $\pi$.

Then $\P$-almost surely, the repelling set $F(\omega) = \{\sigma_1(\omega),\ldots, \sigma_d(\omega)\}$ of $(f^n_\omega)_{n \in \N}$ satisfies $\pi^{-1}(\sigma(\omega)) = F(\omega)$.
\end{prop}

We finish with the Hölder regularity of the unique $\mu$-stationary measure of a proximal random dynamical system $(G, \mu) \curvearrowright S^1$. The original statement in \cite[Theorem 2.3]{holderMeasure} is written for $G \leq \mathrm{Diff}^1(M)$ for any compact smooth manifold $M$, but \cite[Remark 2.10]{holderMeasure} shows that differentiability of the maps in $G$ is not essential: what is truly needed is that all maps of $G$ be biLipschitz.

\begin{teo}[{\cite[Theorem 2.3]{holderMeasure}}] \label{teo:holder}
Consider an action $G \curvearrowright S^1$ by orientation-preserving diffeomorphisms of class $C^1$ with no invariant probability measure on $S^1$, and assume that for some $\delta >0$ the integral
\begin{equation*} 
	\int_{G} \max\left\{\abs{g}_{\mathrm{Lip}}, \abs{g^{-1}}_{\mathrm{Lip}}\right\}^\delta \dd \mu(g)
\end{equation*} is finite.

Then there exist $C, \alpha> 0$ such that any $\mu$-stationary probability  measure $\nu$ on $S^1$ is \emph{$(C,\alpha)$-Hölder continuous}, that is, $\nu(B(x,r)) \leq Cr^\alpha$ for all $x \in S^1$ and $r > 0$.
\end{teo}

\section{Probabilistic Tits alternative in $\mathrm{Homeo}_+(S^1)$}

\begin{proof}[\textbf{Proof of Theorem \ref{teo:b}}] Fix $\mu_1, \, \mu_2$ two nondegenerate probability measures on countable subgroups $G_1,\, G_2$ of $\mathrm{Homeo}_+(S^1)$ that do not preserve a measure on $S^1$. For $i=1,2$, let
\begin{itemize}
	\item $\nu_i$ be the unique $\mu_i$-stationary measure on $S^1$ and $\Lambda_i \subseteq S^1$ the minimal set of $G_i$,
	\item $d_i \in \N$ be the degree of proximality of $G_i \curvearrowright S^1$, and
	\item $\pi_i \colon S^1 \to S^1$ be a factor map of $G_i \curvearrowright S^1$ to a minimal proximal action of $G_i$ such that $\pi_i$ is $d_i$-to-one $\nu_i$-almost everywhere.
\end{itemize} 
 
If $\omega \in \Omega_1$, we write
\begin{itemize}
	\item $(g^n_\omega)_{n \in \N}$ for the random walk driven by $\mu_1$ acting on the image of $\pi_1$ and $\sigma(\omega) \in S^1$ its repelling point, and
	\item $F(\omega) \subset S^1$ for the repelling set of the random walk $(f^n_\omega)_{n \in \N}$.
\end{itemize}
\noindent
When $\omega' \in \Omega_2$ we denote by $g^n_{\omega'},\, \sigma(\omega')$ and $F(\omega')$ the same objects associated to $\mu_2$.
 
Fix once and for all $x \in S^1,\, y \in S^1$ such that $\pi_1^{-1}(x)$ and  $\pi_2^{-1}(y)$ have size $d_1,d_2$ respectively and are disjoint.
\begin{claim}
The following properties are true for $\P_1 \otimes \P_2$-almost every $(\omega, \omega') \in \Omega_1 \times \Omega_2$.
\begin{enumerate}
	\item \label{it:1} The sequence $\{(f^n_\omega(a), f^n_{\omega'}(b))\}_{n \in \N} \subset S^1 \times S^1$ equidistributes with respect to $\nu_1 \otimes \nu_2$ for every $a \in \pi_1^{-1}(x)$ and $b \in \pi_2^{-1}(y)$.
	\item \label{it:2} The equalities $F(\omega) = \pi_1^{-1}(\sigma(\omega))$ and $F(\omega') = \pi_2^{-1}(\sigma(\omega'))$ hold.
	\item \label{it:3} The sets $F(\omega),\, F(\omega'),\, \pi_1^{-1}(x)$ and $\pi_2^{-1}(y)$ are pairwise disjoint.
\end{enumerate}
\end{claim}

\begin{proof}[Proof of the claim:]
\eqref{it:1} The random dynamical system $(G_1 \times G_2, \mu_1 \otimes \mu_2) \curvearrowright S^1 \times S^1$ is locally contracting since $(G_1, \mu_1) \curvearrowright S^1$, $(G_2, \mu_2) \curvearrowright S^1$ are locally contracting. Moreover, for $(x,y) \in S^1 \times S^1$ the orbit $\mathrm{Orb}_{G_1\times G_2}((x,y))$ accumulates on $\Lambda_1 \times \Lambda_2$, so $\Lambda_1 \times \Lambda_2$ is the unique $G_1 \times G_2$-minimal set and Proposition \ref{prop:medidasEstacionarias} shows that $S^1 \times S^1$ has a unique $\mu_1 \otimes \mu_2$-stationary measure, namely $\nu_1 \otimes \nu_2$. Again Proposition \ref{prop:medidasEstacionarias} gives equidistribution $\P_1 \otimes \P_2$-almost surely.

\noindent
\eqref{it:2} This is Proposition \ref{prop:repellingSet}.

\noindent
\eqref{it:3} By independence it suffices to show that $\P_1\left[ z \in F(\omega) \right] = \P_2\left[ z \in F( \omega')\right] = 0$ for any fixed $z \in S^1$, but this follows from $\P_1\left[ z \in F(\omega)\right] = \P_1\left[ \pi_1(z) = \sigma(\omega) \right]$ and the fact that the distribution of $\sigma(\omega)$ is nonatomic.
\end{proof}

We will assume in what follows that $(\omega, \omega') \in \Omega_1\times \Omega_2$ satisfy the previous properties.

Fix an $\epsilon > 0$ and pick $\delta > 0$ such that any interval $I \subset S^1$ with $\abs{I} \leq \delta$ has $\nu_1(I), \nu_2(I) \leq \epsilon$ and also $\nu_1\otimes \nu_2 (D^{\delta}) \leq \epsilon$ where $D \subset S^1\times S^1$ is the diagonal (here $S^1 \times S^1$ is equipped with the $l^\infty$-metric).  Let $\chi = \chi(\omega,\omega') \in (0,\delta/2)$ such that $F(\omega)^\chi$ and $F(\omega')^\chi$ are disjoint.

Suppose that $I\subset S^1$ has diameter at most $\chi$ and $a \in \pi_1^{-1}(x), \, b \in \pi_2^{-1}(y)$. Equidistribution implies that the quantities 
\begin{align} \label{eq:upper}
	\varlimsup_{N \to \infty}\frac{1}{N}\abs{\{0\leq n < N \mid f^n_\omega(a) \in I\}},\quad & \varlimsup_{N \to \infty}\frac{1}{N}\abs{\{0\leq n < N \mid f^n_{\omega'}(b) \in I\}}\\ 
	\text{ and }\qquad \varlimsup_{N \to \infty}\frac{1}{N}|\{0\leq n < N \mid & (f^n_\omega(a),f^n_{\omega'}(b)) \in D^\chi \}| \nonumber
\end{align} are all smaller than $\epsilon$. By considering the different combinations in which the intervals in $f^n_\omega(\pi_1^{-1}(x))^\chi$, $f^n_{\omega'}(\pi_2^{-1}(y))^{\chi}$, $F(\omega)^\chi$ and $ F(\omega')^\chi$ can intersect, we conclude that
\[
	\varlimsup_{N \to \infty}\frac{1}{N}\abs{\{0\leq n < N \mid f^n_\omega(\pi_1^{-1}(x))^\chi,\, f^n_{\omega'}(\pi_2^{-1}(y))^\chi, \, F(\omega)^\chi \text{ and } F(\omega')^\chi \text{ are not pairwise disjoint}\}}
\] is at most $(d_1^2 + d_2^2 + 3d_1d_2)\epsilon$.

Let $\bar{\chi}>0$ such that for $i = 1,2$, the connected components of $\pi_i^{-1}(I)$ have diameter at most $\chi$ if $I \subset S^1$ has diameter at most $\bar{\chi}$. By Theorem \ref{teo:dkn}, $\P_1 \otimes \P_2$-almost surely we can find an $n_0 = n_0(\omega, \omega') \in \N$ such that for all $n \geq n_0$ the inclusions
\[
	g^n_\omega(S^1 - \sigma(\omega)^{\bar{\chi}}) \subseteq g^n_\omega(x)^{\bar{\chi}} \quad \text{ and }\quad g^n_{\omega'}(S^1 - \sigma(\omega')^{\bar{\chi}}) \subseteq g^n_{\omega'}(y)^{\bar{\chi}}
\] hold, so $F(\omega) = \pi_1^{-1}(\sigma(\omega)),\, F(\omega') = \pi_2^{-1}(\sigma(\omega'))$ shows that
\begin{equation} \label{eq:incl}
	f^n_\omega(S^1 - F(\omega)^{\chi}) \subseteq f^n_\omega(\pi_1^{-1}(x))^\chi \quad \text{ and }\quad f^n_{\omega'}(S^1 - F(\omega')^{\chi}) \subseteq f^n_{\omega'}(\pi_2^{-1}(y))^\chi.
\end{equation} Equation \eqref{eq:incl} implies that every $n$ in the set
\[
	\cN = \{ n\geq n_0 \mid f^n_\omega(\pi_1^{-1}(x))^\chi,\, f^n_{\omega'}(\pi_2^{-1}(y))^\chi, \, F(\omega)^\chi \text{ and } F(\omega')^\chi \text{ are pairwise disjoint}\}
\] is such that $f^n_\omega, f^n_{\omega'}$ are a ping-pong pair. Thus
\begin{align*}
	\varliminf_{N \to \infty}\frac{1}{N}\abs{\{0\leq n < N \mid f^n_\omega,\, f^n_{\omega'} \text{ are a ping-pong pair}\}} & \geq \varliminf_{N \to \infty} \frac{1}{N}\abs{\cN\cap [0,N]} \\
	& \geq 1-(d_1^2 + d_2^2 + 3d_1d_2)\epsilon,
\end{align*} and since $\epsilon > 0$ was arbitrary the conclusion follows.
\end{proof}

\section{Probabilistic Tits alternative in $\mathrm{Diff}^1_+(S^1)$}
\subsection*{Preliminary statements} In this subsection $G$ is a countable subgroup of $\mathrm{Diff}^1_+(S^1)$ that acts proximally on $S^1$ and $\mu$ is a nondegenerate probability measure on $G$ such that
\begin{equation} \label{eq:momentCondition} 
	\text{there exists }\delta > 0 \text{ so that }\int_{G} \max\left\{\abs{g}_{\mathrm{Lip}}, \abs{g^{-1}}_{\mathrm{Lip}}\right\}^\delta \dd \mu(g) \text{ is finite.} \tag{M}
\end{equation}

We now state and prove Theorem \ref{teo:mean}, which gives (uniform) exponential contractions in mean in our context. It is a variation on similar statements that have appeared independently in \cite[Proposition 4.18]{KG} and \cite[Theorem 1.3]{GScontracting}, assuming that $\mu$ has finite support in $\mathrm{Diff}^1_+(S^1)$. The proof follows \cite[Proposition 4.18]{KG} closely, along with additional input from \cite[Proposition 4.5]{BarrientosMalicet2024}. This is the only point in the proof where we use that $\mu$ is supported in $\mathrm{Diff}^1_+(S^1)$, namely to obtain inequality \eqref{eq:meanLocal} below.

\begin{teo}[{\cite[Proposition 4.5]{BarrientosMalicet2024}}] \label{teo:meanLocal}
If $\mu$ satisfies the condition \eqref{eq:momentCondition}, there exist $r,\lambda > 0,\, s_0 \in (0,1]$ and $k \in \N_+$ such that
\begin{equation} \label{eq:meanLocal}
	\E\left[ d(f^{k_1}_\omega(x), f^{k_1}_\omega(y))^s \right] \leq e^{-\lambda} d(x,y)^s
\end{equation} for all $x,y \in S^1$ such that $d(x,y) \leq r$ and all $s \in (0, s_0]$.
\end{teo}

\begin{teo} \label{teo:mean}
There exists $\lambda_+ > 0,\, s \in (0,1]$ and $N \in \N$ such that for all $n \geq N$ we have 
\[
	\sup_{x\neq y \in S^1} \E\left[ \frac{d(f^n_\omega(x), f^n_\omega(y))^s}{d(x,y)^s}\right] \leq e^{-\lambda_+ n}.
\]
    
In particular, for all $n \geq N$ we have
\[
	\sup_{x,y \in S^1} \E\left[ d(f^n_\omega(x), f^n_\omega(y)) \right] \leq e^{-\lambda_+ n}
\]
\end{teo}

\begin{proof}
Let $r, \lambda > 0$, $s_0 \in (0,1]$ and $k_1 = k \in \N_+$ given by Theorem \ref{teo:meanLocal}, so \eqref{eq:meanLocal} holds for all $x,y \in S^1$ with $d(x,y) \leq r$ and all $s \in (0,s_0]$.
\begin{claim}
For every $\epsilon_1, \epsilon_2 > 0$ there exists $k_2 \in \N_+$ such that
\[
	\P\left[ d(f^{k_2}_\omega(x), f^{k_2}_\omega(y)) < \epsilon_1 \right] > 1 - \epsilon_2
\] for all $x,y \in S^1$.
\end{claim}

\begin{proof}[Proof of the claim]
This is \cite[Lemma 4.23]{KG}, but we give the proof for completeness.

Let $l \in \N_+$ be large enough so that the points $x_j = j/l \in S^1,\, 0 \leq j \leq l-1$ satisfy
\[
	\P\left[\sigma(\omega) \in (x_j, x_{j+1})\right] \leq \epsilon_2/4
\] for each $j$. When $0 \leq j \leq l-1$ denote by $I_j$ the open interval with endpoints $x_j,\, x_{j+1}$ and length $1 - 1/l$. Choose $k_2 \in \N$ large enough so that for every $0 \leq j \leq l-1$, we have
\[
	\P\left[ \mathrm{diam}(f^n_\omega(I_j)) \leq \epsilon_1 \mid \sigma(\omega) \in (x_j, x_{j+1})\right] \geq 1- \epsilon_2/2.
\]

If $x, y \in S^1$, then $x,y \in I_j$ for all indices $0\leq j \leq l-1$ except at most two values $j_1, j_2$. Thus
\begin{align*}
	\P\left[ d(f^{k_2}_\omega(x), f^{k_2}_\omega(y)) < \epsilon_1 \right] & \geq \P\left[ x, y \in I_{j(\omega)} \text{ where } \sigma(\omega) \in (x_{j(\omega)}, x_{j(\omega) + 1}) \right]\\
									      & \geq \sum_{\substack{j = 0\\ j \neq j_1,j_2}}^{l-1}\P\left[ \mathrm{diam}(f^n_\omega(I_j)) \leq \epsilon_1 \mid \sigma(\omega) \in (x_j, x_{j+1})\right]\P\left[\sigma(\omega) \in (x_j, x_{j+1}) \right] \\ 
									      & \geq (1 - \epsilon_2/2)(1 - 2 \epsilon_2/4) = (1 - \epsilon_2/2)^2 \geq 1 - \epsilon_2. \qedhere
\end{align*}
\end{proof}

Fix $s \in (0,s_0]$ and find $k_2 \in \N_+$ such that
\[
	\P\left[ d(f^{k_2}_\omega(x), f^{k_2}_\omega(y)) < r/4^{1/s} \right] > 1 - r^s/4
\] for all $x,y \in S^1$. If we take $x,y \in S^1$ such that $d(x,y) \geq r$, then
\begin{align*}
	\E\left[ d(f^{k_2}_\omega(x), f^{k_2}_\omega(y))^s \right]  & \leq \P\left[ d(f^{k_2}_\omega(x), f^{k_2}_\omega(y)) > r/4^{1/s} \right] + \left(\frac{r}{4^{1/s}}\right)^s P\left[ d(f^{k_2}_\omega(x), f^{k_2}_\omega(y)) < r/4^{1/s} \right] \\
	& \leq \frac{r^s}{4} + \frac{r^s}{4}  \leq \frac{1}{2} d(x,y)^s.
\end{align*}

For every $k \in \N_+$ and $x,y \in S^1$, define a random variable $K_{k}(\omega) \in \N_+$ (which depends on $x,y$) as follows: if $d(x,y) \leq r$ (resp. $d(x,y) > r$) apply $k_1$ (resp. $k_2$) random iterations of $\omega = (f_{\omega_n})_{n \in \N}$ to the pair $x,y$. Repeat the process on the pair $f^{k_1}_\omega(x), f^{k_1}_\omega(y)$ (resp. $f^{k_2}_\omega(x), f^{k_2}_\omega(y)$) applying iterations of $(f_{\omega_{n + k_1}})_{n \in \N}$ (resp. $(f_{\omega_{n + k_2}})_{n \in \N}$)  until the total number of iterations exceeds $k$ for the first time. By definition we have $k \leq K_{k} \leq k +\max\{k_1,k_2\}$, and by the strong Markov property we see that
\begin{equation} \label{eq:randomContraction}
	\E\left[ d\left(f^{K_{k}(\omega)}_\omega(x), f^{K_{k}(\omega)}_\omega(y)\right)^s \right] \leq  \max\left\{\left(\frac{1}{2}\right)^{\frac{k}{k_2}}, e^{-\frac{k}{k_1}\lambda}\right\}d(x,y)^s.
\end{equation}
	
The $f_{\omega_i}$ are independent and distributed along $\mu$, and hence
\begin{align} \label{eq:indep} \nonumber
	\E\left[\abs{\left(f_{\omega_{K_k}} \circ \cdots \circ f_{\omega_k}\right)^{-1} }^{s}_\mathrm{Lip}\right] & \leq \E\left[ \abs{f_{\omega_{K_k}}^{-1}}^s_\mathrm{Lip} \cdots \abs{f_{\omega_k}^{-1}}^s_\mathrm{Lip}\right] \\ 
	\nonumber  & \leq \E \left[ \abs{f_{\omega_{k + \max\{k_1,k_2\} }}^{-1}}^s_\mathrm{Lip} \cdots \abs{f_{\omega_k}^{-1}}^s_\mathrm{Lip}\right] \\ & = \int_G \abs{g^{-1}}^{s\max\{k_1, k_2\}}_\mathrm{Lip} \dd \mu(g).
\end{align} We deduce that 
\begin{align*}
	\E\left[ d(f^{k}_\omega(x), f^{k}_\omega(y))^{s/2} \right] & \leq \E\left[\abs{\left(f_{\omega_{K_k}} \circ \cdots \circ f_{\omega_k}\right)^{-1} }^{s/2}_\mathrm{Lip} d\left(f^{K_{k}(\omega)}_\omega(x), f^{K_{k}(\omega)}_\omega(y)\right)^{s/2}\right] \\
	& \leq \E\left[\abs{\left(f_{\omega_{K_k}} \circ \cdots \circ f_{\omega_k}\right)^{-1} }^{s}_\mathrm{Lip}\right]^{1/2} \E\left[ d\left(f^{K_{k}(\omega)}_\omega(x), f^{K_{k}(\omega)}_\omega(y)\right)^{s}\right]^{1/2} \\
	& \leq \left(\int_G \abs{g^{-1}}^{s\max\{k_1,k_2\}}_\mathrm{Lip} \dd \mu(g) \right)^{1/2} \max\left\{\left(\frac{1}{2}\right)^{\frac{k}{2 k_2}}, e^{-\frac{k}{2 k_1}\lambda}\right\}d(x,y)^{s/2},
\end{align*} where we have used \eqref{eq:randomContraction} and \eqref{eq:indep} in the last inequality. If $s \leq \delta/\max\{k_1,k_2\}$, where $\delta > 0$ is provided by the condition \eqref{eq:momentCondition}, the term $\int_G \abs{g^{-1}}^{s\max\{k_1,k_2\}}_\mathrm{Lip} \dd \mu(g)$ is also finite. But the term $\max\left\{\left(\frac{1}{2}\right)^{\frac{k}{2 k_2}}, e^{-\frac{k}{2 k_1}\lambda}\right \} $ converges to 0 as $k \to \infty$, and hence there exists $k \in \N_+$ and $\lambda > 0$ such that
\[
	\E\left[ d(f^k_\omega(x), f^k_\omega(y))^s\right] \leq e^{-\lambda} d(x,y)^s
\] for all $x,y \in S^1$ and $ 0 < s \leq \min\{s_0/2, \delta/(2k_1), \delta/(2k_2)\}$. By the Markov property, for all $n \in \N$ we have
\[
	\E\left[ d(f^n_\omega(x), f^n_\omega(y))^s\right] \leq e^{-\lambda \floor{n/k}} d(x,y)^s \leq e^{\lambda} e^{-\lambda n/k}d(x,y)^s.
\] The conclusion follows by setting $\lambda_+ = \lambda/(2k)$ and choosing $N \in \N$ so that $e^{\lambda} e^{-\lambda N /(2k)} < 1$.
\end{proof}
 
\begin{rmk}
The previous theorem says that $(G,\mu) \curvearrowright S^1$ is \emph{$\mu$-contracting}, according to the terminology of Benoist-Quint in \cite[Section 11.1]{BQbook}. As a consequence, all of the limit laws available for cocycles in this setting (that is, the central limit theorem, the law of the iterated logarithm and large deviations estimates, see \cite[Section 12.1]{BQbook}) hold in this setting. In particular the Lyapunov cocycle $(g,x) \mapsto \log g'(x)$ satisfies these limit laws. This recovers \cite[Theorem 1.14]{gelfertSalcedo2024}, for instance. However, \cite[Theorem 12.1]{BQbook} requires the cocycle to be Lipschitz with integrable Lipschitz constant, but the proofs go through without relevant changes if the Lipschitz condition is replaced by a $\tau$-Hölder one.
\end{rmk}

Denote by $(\overline{f}^n_\omega)_{n \in \N}$ the \emph{left (or inverse) random walk} $\overline{f}^n_\omega = f_{\omega_0} \circ f_{\omega_1} \circ \cdots f_{\omega_n}$. Define the random variable $T(\omega) \in S^1$ as the repulsor of the random walk $\left(f_{\omega_n}^{-1} \circ f_{\omega_{n-1}}^{-1} \circ \cdots \circ f_{\omega_0}^{-1}\right)_{n \in \N}$. The following is an analogue of \cite[Theorem 4.16]{aoun}.
\begin{prop} \label{prop:convExp}
Let $\lambda_+ > 0$ be the constant provided by Theorem \ref{teo:mean}. There exists $\lambda_- > 0,\, N \in \N$ such that
\begin{equation} \label{eq:convExpEn}
	\sup_{x \in S^1}\E\left[d\left((f^n_\omega)^{-1}(x), \sigma(\omega)\right)\right] \leq e^{-\lambda_- n}
\end{equation} and 
\begin{equation} \label{eq:convExpEn2}
	\sup_{x \in S^1} \E \left[ d\left(\overline{f}^n_\omega(x), T(\omega)\right) \right] \leq e^{-\lambda_+ n}
\end{equation} hold for $n \geq N$. Moreover, there exist $C_-, \alpha_-> 0$ such that the distribution of $T$ is $(C_-,\alpha_-)$-Hölder continuous.
\end{prop}

\begin{proof}
By applying \eqref{eq:convExpEn} to the random walk on $\mathrm{Diff}^1_+(S^1)$ driven by $\overline{\mu}$ where $\overline{\mu}(g) = \mu(g^{-1})$ for all $g \in \mathrm{Diff}^1_+(S^1)$ we conclude that \eqref{eq:convExpEn2} holds. The Hölder continuity of $T$ also follows from Theorem \ref{teo:holder} since $\overline{\mu}$ also satisfies the assumption \eqref{eq:momentCondition}, so it suffices to prove \eqref{eq:convExpEn}.

Let $n,k \in \N$ with $0<n<k$ and let $x,y \in S^1$. We have that
\[
	\E\left[d\left((f^n_\omega)^{-1}(x), \sigma(\omega)\right)\right] \leq \E\left[d\left((f^n_\omega)^{-1}(x), (f^k_\omega)^{-1}(y)\right)\right] + \E\left[d\left((f^k_\omega)^{-1}(y), \sigma(\omega)\right)\right].
\] The random walk driven by $\overline{\mu}$ gives a $\lambda_- > 0$ such that
\[
	\sup_{u,v \in S^1} \E\left[ d\left( f_{\omega_n}^{-1} \circ \cdots \circ f_{\omega_0}^{-1}(u), f_{\omega_n}^{-1} \circ \cdots \circ f_{\omega_0}^{-1}(v)\right) \right] \leq e^{-\lambda_- n}
\] for all sufficiently large $n \in \N$. In particular we deduce that
\begin{align*}
	\E\left[d\left((f^n_\omega)^{-1}(x), (f^k_\omega)^{-1}(y)\right)\right] & = \int \E\left[d\left((f^n_\omega)^{-1}(x), (f^n_\omega)^{-1} \circ \gamma^{-1} (y)\right)\right] \dd \mu^{\ast(k-n)}(\gamma) \\
	& \leq \sup_{u,v \in S^1}\E\left[d\left((f^n_\omega)^{-1}(u), (f^n_\omega)^{-1}(v)\right)\right] \\
	& = \sup_{u,v \in S^1} \E\left[ d\left( f_{\omega_n}^{-1} \circ \cdots \circ f_{\omega_0}^{-1}(u), f_{\omega_n}^{-1} \circ \cdots \circ f_{\omega_0}^{-1}(v)\right) \right] \leq e^{-\lambda_- n}
\end{align*} where we have used that the $f_{\omega_j}$ are independent and identically distributed in the last equality.  Hence the inequality
\begin{equation} \label{eq:convExp}
	\sup_{x \in S^1} \E\left[d\left((f^n_\omega)^{-1}(x), \sigma(\omega)\right)\right] \leq e^{-\lambda_- n} + \E\left[d\left((f^k_\omega)^{-1}(y), \sigma(\omega)\right)\right]
\end{equation} holds, and by integrating \eqref{eq:convExp} in $\dd \nu(y)$ we conclude that
\begin{align*}
	\sup_{x \in S^1} \E\left[d\left((f^n_\omega)^{-1}(x), \sigma(\omega)\right)\right] &\leq e^{-\lambda_- n} + \E\left[\int_{S^1}d\left((f^k_\omega)^{-1}(y), \sigma(\omega)\right) \dd \nu(y)\right]\\
	& = e^{-\lambda_- n} + \E\left[\int_{S^1}d(y, \sigma(\omega)) \dd (f^k_\omega)^{-1} \nu(y)\right].
\end{align*} The dominated convergence theorem and Theorem \ref{teo:dkn}, \eqref{it:dkn2} imply that
\[
	\E\left[\int_{S^1}d(y, \sigma(\omega)) \dd (f^k_\omega)^{-1} \nu(y)\right] \xrightarrow[k \to \infty]{} \E\left[ \int_{S^1}d(y,\sigma(\omega)) \dd 1_{\sigma(\omega)}(y)\right] = 0,
\] so \eqref{eq:convExpEn} holds.
\end{proof}

Recall that the proof of Theorem \ref{teo:b} (when the subgroups of $\mathrm{Homeo}_+(S^1)$ act proximally) involves trying to find for a given $n \in \N$ small disjoint open intervals $U, V \subset S^1$ containing $\sigma(\omega)$ and $f^n_\omega(0)$ respectively such that $f^n_\omega(S^1 \setminus U) \subseteq V$. Here, the diameters of $U,V$ depend on $\omega$ but not on $n$. In this sense, $\sigma(\omega)$ and $f^n_\omega(0)$ are a repulsor-attractor pair for $f^n_\omega$ in a weak sense that is sufficient for the proof of the qualitative statement in Theorem \ref{teo:b}. On the other hand, to show Theorem \ref{teo:a} we need to show that $\P\left[f^n_\omega(S^1 \setminus U_n) \subset V_n\right]$ is exponentially close to 1 as $n$ increases, where $U_n$ and $V_n$ are disjoint intervals centered around $\sigma(\omega)$ and $f^n_\omega(0)$ respectively such that $\mathrm{diam}(U_n),\, \mathrm{diam}(V_n)$ are exponentially small in $n$. The fact that this contraction takes place does not follow from the definition of $\sigma(\omega)$, since $\sigma(\omega)$ is defined by an asymptotic condition saying that $\P$-almost surely any fixed closed interval inside $S^1 \setminus \{\sigma(\omega)\}$ is eventually contracted by $f^n_\omega$. Nevertheless, the following proposition shows that $\sigma(\omega)$ and $f^n_\omega(0)$ are a repulsor-attractor pair for $f^n_\omega$ in a strong sense suitable for our purposes. This is one of the main points where the strategy deviates from the linear case.

\begin{prop} \label{prop:contraccionExp}
There exists an $\epsilon \in (0,1)$ such that
\[
	\limsup_{n \to \infty} \frac{1}{n}\log \P\left[ f^n_\omega\left(S^1 - \sigma(\omega)^{\epsilon^n}\right) \text{is not included in }f^n_\omega(0)^{\epsilon^n}\right] < 0.
\]
\end{prop}

\begin{proof}
Let $\lambda >0,\, N \in \N$ be the constants given by Theorem \ref{teo:mean}, so
\begin{equation} \label{eq:mean}
	\sup_{x,y \in S^1} \E\left[d\left(f^n_\omega(x), f^n_\omega(y)\right)\right] \leq e^{-\lambda n}
\end{equation} for all $n \geq N$.

For $n \geq N$ define $K_n = \floor{e^{\lambda n /3}}$,  the grid $x^n_k = k/K_n \in S^1,\, 0\leq k \leq K_n-1$ and the event 
\[
	V_n = \{\omega \in \Omega \mid d(f^n_\omega(x^n_k), f^n_\omega\left(x^n_k\right)) \leq e^{-\lambda n/2} \text{ for all } 0 \leq k \leq K_n-1\},
\] so we have
\begin{equation*}
	\P\left[V_n^c\right] \leq \sum_{k = 1}^{K_n} \P\left[ d\left( f^n_\omega(x_k), f^n_\omega(x_{k+1})\right) \geq e^{-\lambda n/2} \right] \leq e^{-\lambda n/2} K_n \leq e^{-\lambda n/6},
\end{equation*} where we have used the Markov inequality and \eqref{eq:mean}.

Notice that if $\omega \in V_n$, then there exists a unique interval $J_{n,\omega} \subset S^1$ of the form $[x^n_{j},x^n_{j + 1})$ such that $\mathrm{diam}\left(f^n_\omega (J_{n, \omega})\right) \geq 1 - e^{-\lambda n/2}$.

\begin{claim} There exists $C > 0$ such that $\E\left[d(J_{n,\omega}, \sigma(\omega)) \mid V_n\right] \leq C e^{-\lambda_1 n}$ for all large enough $n \in \N$, where $\lambda_1 = \max\{ \lambda/6, \lambda_-\}$ and $\lambda_-$ is given by Proposition \ref{prop:convExp}.
\end{claim}

\begin{proof}[Proof of the claim]
Given $\omega \in V_n$, define $j_{n,\omega} \in S^1$ by
\[
	j_{n,\omega} = \begin{cases}(f^n_\omega)^{-1}(0) & \text{ if }0 \in f^n_\omega(J_{n,\omega})\\
	(f^n_\omega)^{-1}(1/2) & \text{ otherwise.}
	\end{cases}
\] Since $\mathrm{diam}\left( f^n_\omega(J_{n,\omega})\right) \geq 1 - e^{-\lambda n/2}$, for all $n > 2\log 2/\lambda$ we have that $f^n_\omega(J_{n,\omega})$ contains 0 or $1/2$, so $j_{n, \omega} \in J_{n,\omega}$. Thus
\begin{align*}
	\E\left[d(J_{n,\omega}, \sigma(\omega)) \mid V_n\right] & \leq \E\left[\mathrm{diam}(J_{n, \omega}) + d(j_{n,\omega}, \sigma(\omega)) \mid V_n\right]\\ &\leq \frac{1}{K_n} + \E\left[d\left((f^n_\omega)^{-1}(0), \sigma(\omega)\right) \mid V_n\right] + \E\left[d\left((f^n_\omega)^{-1}(1/2), \sigma(\omega)\right) \mid V_n\right] \\
	& \leq \frac{1}{K_n} + \P\left[ V_n \right]^{-1}\left( \E\left[d\left((f^n_\omega)^{-1}(0), \sigma(\omega)\right)\right] + \E\left[d\left((f^n_\omega)^{-1}(1/2), \sigma(\omega)\right)\right]\right)
\end{align*} for $n > \max\{ 2 \log 2/\lambda, N\}$. From the bound \eqref{eq:convExpEn} and the fact that $\P[V_n]$ is bounded away from 0 we obtain the conclusion.
\end{proof}
    
Let $C, \lambda_1 >0 $ be the constants given by the previous claim and take $\epsilon \in (e^{-\lambda_1},1)$, so
\begin{align} \nonumber
	\P\left[ f^n_\omega\left(S^1 \setminus \sigma(\omega)^{\epsilon^n}\right) \not \subseteq f^n_\omega(0)^{\epsilon^n} \right] & \leq \P\left[ f^n_\omega\left(S^1 \setminus \sigma(\omega)^{\epsilon^n}\right) \not \subseteq f^n_\omega(0)^{\epsilon^n} \,\middle|\, V_n \right]  + \P\left[ V_n^c \right]\\
	& \leq  \P\left[J_{n,\omega} \not \subseteq \sigma(\omega)^{\epsilon^n} \text{ or }f^n_\omega\left(S^1\setminus J_{n,\omega}\right) \not \subseteq f^n_\omega(0)^{\epsilon^n}\,\middle|\, V_n \right] + \P\left[ V_n^c \right]. \label{eq:epsilon1}
\end{align} for large enough $n \in \N$. Since $\epsilon > e^{-\lambda/3}$ and $K_n = \floor{e^{\lambda n/3}}$, there exists a constant $C'> 0$ such that the inequalities
\begin{align} \nonumber
	\P\left[ J_{n,\omega} \not \subseteq \sigma(\omega)^{\epsilon^n} \, \middle| \, V_n \right] & \leq \P\left[ d(J_{n,\omega} , \sigma(\omega)) \geq \epsilon^n - \mathrm{diam}(J_{n,\omega}) \, \middle| \, V_n \right] \\ 
	& \leq \frac{\E\left[ d(J_{n,\omega}, \sigma(\omega)\, \middle| \, V_n \right]}{\epsilon^n - 1/K_n} \leq C \left( \frac{e^{-\lambda_1}}{\epsilon} \right)^n\left(\frac{1}{1 - 1/(\epsilon^nK_n)}\right) \leq C' \left( \frac{e^{-\lambda_1}}{\epsilon} \right)^n  \label{eq:epsilon2}
\end{align} hold, so the right-hand side of \eqref{eq:epsilon2} decreases with exponential speed towards 0 by the choice of $\epsilon$.

Similarly, since $\epsilon^n \geq e^{-\lambda n/6}  \geq \mathrm{diam}\left( f^n_\omega\left( S^1 \setminus J_{n,\omega} \right) \right)$ for all large $n \in \N$, we see that
\begin{align}
	\nonumber \P\left[ f^n_\omega\left( S^1 \setminus J_{n,\omega} \right) \not \subseteq f^n_\omega(0)^{\epsilon^n} \, \middle| \, V_n \right] & \leq \P \left[ 0 \in J_{n,\omega} \, \middle| \, V_n \right] \\
	\nonumber & \leq \P \left[ 0 \in J_{n,\omega} \text{ and } J_{n,\omega} \subseteq \sigma(\omega)^{\epsilon^n} \, \middle| \, V_n \right] + C' \left( \frac{e^{-\lambda_1}}{\epsilon} \right)^n \\
	\nonumber & \leq \P \left[ d(0, \sigma(\omega) ) \leq \epsilon^n \, \middle| \, V_n \right]  + C' \left( \frac{e^{-\lambda_1}}{\epsilon} \right)^n, \label{eq:epsilon3}
\end{align} where we have used \eqref{eq:epsilon2} in the second inequality. Moreover, Theorem \ref{teo:holder} provides $C'',\alpha > 0$ such that
\[
	\P\left[ d(0, \sigma(\omega)) \leq \epsilon^n\right] \leq C'' \epsilon^{\alpha n}, 
\] and from \eqref{eq:epsilon1} we conclude that
\begin{equation} \label{eq:epsilon4}
	\P\left[ f^n_\omega\left( S^1 \setminus J_{n,\omega} \right) \not \subseteq f^n_\omega(0)^{\epsilon^n} \, \middle| \, V_n \right] \leq \P\left[ V_n \right]^{-1}C'' \epsilon^{\alpha n} + C' \left( \frac{e^{-\lambda_1}}{\epsilon} \right)^n.
\end{equation}
The bounds \eqref{eq:epsilon4} and \eqref{eq:epsilon2} show that the right-hand side of \eqref{eq:epsilon1} is exponentially small in $n$. This finishes the proof of the proposition.
\end{proof}

From now on, the rest of the proof of Theorem \ref{teo:a} follows the strategy of \cite{aoun}.

\begin{prop} \label{prop:auto}
For every $t \in (0,1)$ we have
\[
	\limsup_{n\to \infty} \frac{1}{n}\log \P\left[ d(f^n_\omega(0),\sigma(\omega)) \leq t^n \right] < 0.
\]
\end{prop}

\begin{proof}
We start with a version of \cite[Theorem 4.35]{aoun} (which in turn is inspired by \cite[Lemme 8]{guivarch1990}), which states that the variables $f^n_\omega(0)$ and $\sigma(\omega)$ become asymptotically independent with exponential speed as $n \to \infty$.
\begin{claim}
There exists a random variable $S(\omega) \in S^1$ independent of $\sigma$ and constants $C, \lambda > 0$ such that for any Lipschitz function $\psi \colon S^1 \times S^1 \to \R$ we have
\[
	\abs{\E\left[\psi\left(f^n_\omega(0),\sigma(\omega)\right)\right] -  \E\left[\psi\left(\sigma(\omega), S(\omega)\right)\right]} \leq C e^{-\lambda n} \abs{\psi}_{\mathrm{Lip}}
\] for sufficiently large $n \in \N$, where 
\[
	\abs{\psi}_{\mathrm{Lip}} = \sup_{\substack{x,y,u,v \in S^1\\ x \neq y \text{ \emph{or} } u \neq v}}\frac{\abs{\psi(x,u) - \psi(y,v)}}{d(x,y) + d(u,v)}.
\]
\end{claim}

\begin{proof}[Proof of the claim]
Let $\lambda_-, \lambda_+ > 0$ be the constants given by Proposition \ref{prop:convExp}. Consider an independent copy $\omega' = (f_{\omega_n'})_{n \in \N}$ of the process $\omega$ (that is, a coupling of $\P$ with itself). Define $S(\omega')$ as the repulsor of the random walk $\left(f_{\omega_n'}^{-1} \circ f_{\omega_{n-1}'}^{-1} \circ \cdots \circ f_{\omega_0'}^{-1}\right)_{n \in \N}$, so
\[
	\sup_{x \in S^1} \E\left[ d\left( \overline{f}^n_{\omega'}(x), S(\omega')\right)\right] \leq e^{-\lambda_+ n}
\] holds for all large $n \in \N$.

Decompose
\[
	\abs{\E\left[\psi\left(f^n_\omega(0),\sigma(\omega)\right)\right] -  \E\left[\psi\left(\sigma(\omega), S(\omega')\right)\right]} \leq \triangle_1 + \triangle_2 + \triangle_3 + \triangle_4
\] where
\begin{itemize}
	\item $\triangle_1 = \abs{\E\left[\psi\left(\sigma(\omega), f^n_\omega(0)\right)\right] - \E\left[\psi\left((f^n_\omega)^{-1}(0), f^n_\omega(0)\right)\right]}$,
	\item $\triangle_2 = \abs{\E\left[\psi\left((f^n_\omega)^{-1}(0), f^n_\omega(0)\right)\right] - \E\left[\psi\left(\left(f^{\ceil{n/2}}_\omega\right)^{-1}(0), f_{\omega_n} \circ \cdots \circ f_{\omega_{\ceil{n/2} + 1}}(0)\right)\right]}$,
	\item $\triangle_3 = \abs{\E\left[\psi\left(\left(f^{\ceil{n/2}}_\omega\right)^{-1}(0), f_{\omega_n} \circ \cdots \circ f_{\omega_{\ceil{n/2} + 1}}(0)\right)\right] - \E\left[\psi\left(\sigma(\omega),\overline{f}^{\floor{n/2}}_{\omega'}(0)\right)\right] }$
    
\noindent
\quad \;  $= \abs{\E\left[\psi\left(\left(f^{\ceil{n/2}}_\omega\right)^{-1}(0), \overline{f}^{\floor{n/2}}_{\omega'}(0)\right)\right]- \E\left[\psi\left(\sigma(\omega),\overline{f}^{\floor{n/2}}_{\omega'}(0)\right)\right] }$, and
	\item $\triangle_4 = \abs{\E\left[\psi\left(\sigma(\omega),\overline{f}^{\floor{n/2}}_{\omega'}(0)\right)\right] - \E\left[\psi\left(\sigma(\omega), S(\omega')\right)\right]}$.
\end{itemize}

Proposition \ref{prop:convExp} shows that $\triangle_1 \leq \abs{\psi}_\mathrm{Lip}e^{-\lambda_- n}, \, \triangle_3 \leq \abs{\psi}_\mathrm{Lip} e^{-\lambda_- n/2}$ and $\triangle_4 \leq \abs{\psi}_\mathrm{Lip} e^{-\lambda_+n/2}$ for all large $n \in \N$. Moreover
\begin{align*}
	\triangle_2 &\leq \abs{\psi}_\mathrm{Lip}\left(\E\left[d\left((f^n_\omega)^{-1}(0), \left(f^{\ceil{n/2}}_\omega\right)^{-1}(0)\right)\right] + \E\left[ d\left(f^n_\omega(0), f_{\omega_n} \circ \cdots \circ f_{\omega_{\ceil{n/2} + 1}}(0)\right)\right]\right)\\
	&= \abs{\psi}_\mathrm{Lip}\left(\E\left[d\left((f^n_\omega)^{-1}(0), \left(f^{\ceil{n/2}}_\omega\right)^{-1}(0)\right)\right] + \E\left[ d\left(\overline{f}^n_\omega(0), \overline{f}^{\floor{n/2}}_\omega(0)\right)\right]\right)\\
	& \leq  \abs{\psi}_\mathrm{Lip}\left(\E\left[d\left((f^n_\omega)^{-1}(0), \sigma(\omega) \right) \right] + \E \left[ d\left(\left(f^{\ceil{n/2}}_\omega\right)^{-1}(0), \sigma(\omega) \right)\right] \right. \\
	&  \hspace{4cm} + \left. \E\left[ d\left(\overline{f}^n_\omega(0), T(\omega)\right)\right] + \E\left[ d\left( \overline{f}^{\floor{n/2}}_\omega(0), T(\omega)\right) \right]\right)
\end{align*}
which is at most 
\[
	\abs{\psi}_\mathrm{Lip} \left( e^{-\lambda_- n} + e^{-\lambda_- n/2} + e^{-\lambda_+ n} + e^{-\lambda_+ n/2}\right)
\] by Proposition \ref{prop:convExp} again. The claim follows.
\end{proof}

For any $\epsilon \in (0,1/2)$, take $\phi_\epsilon \colon [0,1] \to [0,1]$ a $1/\epsilon$-Lipschitz function such that $\restr{\phi}{[0,\epsilon]} = 1$ and $\restr{\phi}{[2\epsilon,1]} = 0$, so
\[
	1_{[0,\epsilon]} \leq \phi_\epsilon \leq 1_{[0,2\epsilon]}
\] holds and $\psi_\epsilon \doteq \phi_\epsilon \circ d \colon S^1 \times S^1 \to [0,1]$ is also $1/\epsilon$-Lipschitz. Let $C,\lambda > 0$ be the constants given by the previous claim.

Now for all large $n \in \N$ we have
\begin{align*}
	\P\left[ d\left(f^n_\omega(0), \sigma(\omega)\right) \leq t^n \right] \leq \E\left[\psi_{t^n}\left(f^n_\omega(0), \sigma(\omega)\right)\right] & \leq \E\left[\psi_{t^n}\left(\sigma(\omega), S(\omega)\right)\right] + Ce^{-\lambda n} \abs{\psi_{t^n}}_\mathrm{Lip} \\ & \leq  \P\left[ d\left(\sigma(\omega), S(\omega)\right) \leq 2t^n\right] + Ce^{-\lambda n} \abs{\psi_{t^n}}_\mathrm{Lip} \\
	& \leq \sup_{x  \in S^1}\P\left[d(\sigma(\omega),x) \leq 2 t^n\right] + Ce^{-\lambda n} \abs{\psi_{t^n}}_\mathrm{Lip},
\end{align*} where we have used the independence of $\sigma$ and $S$ in the last inequality. The first term
\[
	\sup_{x  \in S^1}\P\left[ d(\sigma(\omega),x) \leq 2 t^n\right]
\] is exponentially small in $n$ by Theorem \ref{teo:holder}, and the second term
\[
	Ce^{-\lambda n/2} \abs{\psi_{t^n}}_\mathrm{Lip} = C\left( \frac{e^{-\lambda/2}}{t}\right)^n
\] is also exponentially small in $n$ whenever $t  > e^{-\lambda/2}$. The proposition is thus proven in this case, and is also true for $t \leq  e^{-\lambda /2}$ as a consequence.
\end{proof}

\begin{rmk}
In the proof of the previous theorem we have abused notation by writing $S(\omega)$: the random variable $S$ is not a function of $\omega$, and is defined on a larger probability space. We have done so as to not to weigh down the notation with distinctions between this larger probability space and its quotient $(\Omega, \P)$, since all relevant means and measures of sets coincide with those of $\P$.
\end{rmk}

\begin{proof}[\textbf{Proof of Theorem \ref{teo:a}}] Fix $\mu_1, \, \mu_2$ two nondegenerate probability measures on countable subgroups $G_1,\, G_2$ of $\mathrm{Diff}^1_+(S^1)$ acting proximally on $S^1$ such that $\mu_1,\, \mu_2$ satisfy the moment condition \eqref{eq:momentCondition}.

\begin{claim} \label{prop:inter}
For every $n \in \N$ let $\omega' \in \Omega_2 \mapsto l_n(\omega') \in S^1$ be a measurable map. For every $t \in (0,1)$ we have
\begin{equation} \label{eq:claim1}
	\limsup_{n\to \infty} \frac{1}{n}\log \P_1 \otimes \P_2\left[ d(f^n_\omega(0), l_n(\omega')) \leq t^n \right] < 0
\end{equation} and 
\begin{equation} \label{eq:claim2}
	\limsup_{n\to \infty} \frac{1}{n}\log \P_1 \otimes \P_2\left[ d(\sigma(\omega), l_n(\omega')) \leq t^n \right] < 0.
\end{equation}
\end{claim}

\begin{proof}[Proof of the claim]
By independence, we have 
\begin{align*}
	\P_1 \otimes \P_2\left[ d(f^n_\omega(0), l_n(\omega'))\leq t^n\right] \leq \sup_{x \in S^1}\P_1\left[ d(f^n_\omega(0), x)\leq t^n\right] = \sup_{x \in S^1} \P_1\left[ d(\overline{f}^n_\omega(0), x)\leq t^n \right].
\end{align*}
Proposition \ref{prop:convExp} and the Markov inequality imply that
\[
	\P_1\left[ d\left(\overline{f}^n_\omega(0), T(\omega)\right) \geq e^{-\lambda_+ n/2} \right] \leq e^{-\lambda_+ n/2},
\] for some $\lambda_+ > 0$ and all large $n \in \N$, and thus
\[
	\sup_{x \in S^1} \P_1\left[ d\left(\overline{f}^n_\omega(0), x\right)\leq t^n \right] \leq \sup_{x \in S^1}\P_1\left[ d\left(T(\omega), x\right) \leq t^n + e^{-\lambda_+ n/2}\right] + e^{-\lambda_+ n/2}.
\] Take $C, \alpha > 0$ such that the distribution of $T$ is $(C,\alpha)$-Hölder, so
\[
	\sup_{x \in S^1}\P_1\left[ d(T(\omega), x) \leq t^n + e^{-\lambda_+ n/2} \right] \leq C(t^n + e^{-\lambda_+ n/2})^{\alpha}
\] and
\[
	\P_1 \otimes \P_2\left[ d(f^n_\omega(0), l_n(\omega'))\leq t^n\right] \leq C(t^n + e^{-\lambda_+ n/2})^\alpha + e^{-\lambda_+ n/2}
\] is exponentially small in $n$. This gives \eqref{eq:claim1}, and \eqref{eq:claim2} follows in the same way.
\end{proof}

Take $\epsilon \in (0,1)$ so that the conclusion of Proposition \ref{prop:contraccionExp} is verified for $\P_1 = \mu_1^{\otimes \N}$ and $\P_2 = \mu_2^{\otimes \N}$. Given $\omega \in \Omega_1,\, \omega' \in \Omega_2$ and $n \in \N$ we say that \emph{$f^n_\omega$ and $f^n_{\omega'}$ are in $\epsilon$-transverse position at time $n$} if the intervals $f^n_\omega(0)^{\epsilon^n},\, f^n_{\omega'}(0)^{\epsilon^n}, \,\sigma(\omega)^{\epsilon^n}$ and $\sigma(\omega')^{\epsilon^n}$ are pairwise disjoint. This is exactly the situation in Figure \ref{fig:model} for 
\[
	I_{n,\omega} = f^n_\omega(0)^{\epsilon^n}, \quad I_{n,\omega'} = f^n_{\omega'}(0)^{\epsilon^n}, \quad J_{n,\omega} = \sigma(\omega)^{\epsilon^n}\, \text{ and }\,J_{n,\omega'} = \sigma(\omega')^{\epsilon^n}.\]

Proposition \ref{prop:auto} shows that the probability that the pair $(f^n_\omega(0)^{\epsilon^n},\sigma(\omega)^{\epsilon^n})$ or the pair $(f^n_{\omega'}(0)^{\epsilon^n},\sigma(\omega')^{\epsilon^n})$ intersect is exponentially small in $n$. The previous claim shows that the probability that the remaining pairs 
\[
	(f^n_\omega(0)^{\epsilon^n},\sigma(\omega')^{\epsilon^n}), \quad (f^n_{\omega'}(0)^{\epsilon^n},\sigma(\omega)^{\epsilon^n}), \quad (f^n_{\omega}(0)^{\epsilon^n}, f^n_{\omega'}(0)^{\epsilon^n})\,  \text{ or }\, (\sigma(\omega)^{\epsilon^n},\sigma(\omega')^{\epsilon^n})
\] intersect is exponentially small in $n$. We conclude that 
\[
	\limsup_{n\to \infty} \frac{1}{n}\log \P_1 \otimes \P_2\left[ f^n_\omega \text{ and } f^n_{\omega'} \text{ are not in $\epsilon$-transverse position}\, \right]<0. \qedhere
\]
\end{proof}

\providecommand{\bysame}{\leavevmode\hbox to3em{\hrulefill}\thinspace}

\providecommand{\href}[2]{#2}

\end{document}